\documentclass[12pt]{article}
 \usepackage {pdflscape}
 \usepackage{graphicx}
\usepackage{epsfig}
\usepackage[cmex10]{amsmath}
\usepackage{subfig}
\usepackage{amssymb}
\usepackage{float}
\usepackage{amsthm}
\newsavebox{\savepar}  % use as: \begin{boxit} #1 \end{boxit}
\newenvironment{boxit}{\begin{center} \begin{lrbox}{\savepar}
\begin{minipage}[b]{120mm}}
{\end{minipage}\end{lrbox}\fbox{\usebox{\savepar}} \end{center}}

\newtheorem{prethm}{{\bf Theorem}}

\newtheorem{precor}{{\bf Corollary}}

\newtheorem{preprop}{{\bf Proposition}}

\newtheorem{preque}{{\bf Question}}

\newtheorem{preques}{{\bf Question}}

\newtheorem{prelemma}{{\bf Lemma}}

\newenvironment{lemma}{\begin{prelemma}{\hspace{-0.5
em}{\bf.}}}{\end{prelemma}}
\newtheorem{prefact}{{\bf Fact}}

\newtheorem{preobs}{{\bf Observation}}

\newtheorem{prefig}{{\bf Figure}}

\newtheorem{prelemm}{{\bf Lemma}}

\newtheorem{preex}{{\bf Example}}

\newtheorem{prepro}{{\bf Proposition}}

\newtheorem{prelem}{{\bf Theorem}}

\newtheorem{preconj}{{\bf Conjecture}}

\newtheorem{predeff}{{\bf Definition}}

\def\newpic#1{}
\date{}

\begin{document}

\title{
{\Large{\bf A generalization of line graphs via link scheduling in
wireless networks}}}
%\\ cncm23.tex }}}
%

%\author{Ali Ghiasian$^a$}
%\fnref{label2}
% {ghiasian@ec.iut.ac.ir}
%\fntext[label2]{label2}
%\cortext[cor1]{core1}
%\address{Department of Electrical and Computer Engineering\\
%Isfahan University of Technology\\
%84156-83111, Isfahan, Iran}%\fnref{label3}}
%\fntext[label3]{label3}

\author{Ali Ghiasian$^a$, Behnaz Omoomi$^b$ and Hossein Saidi$^a$}
%\fnref{label2}
%\ead {bomoomi@cc.iut.ac.ir}
%\fntext[label2]{label2}
%\cortext[cor1]{core1}
%\address{
%$^a$Department of Electrical and Computer Engineering\\
%$^b$Department of Mathematical Science\\
%Isfahan University of Technology\\
%84156-83111, Isfahan, Iran}%\fnref}{label3}}
%\fntext[label3]{label3}

%\author{Hossein Saidi$^a$}
%\fnref{label2}
%\ead {hsaidi@cc.iut.ac.ir}
%\fntext[label2]{label2}
%\cortext[cor1]{core1}
\maketitle
\begin{center}
$^a$Department of Electrical and Computer Engineering\\
$^b$Department of Mathematical Science\\
Isfahan University of Technology\\
84156-83111, Isfahan, Iran%\fnref}{label3}}
\end{center}
%\address{Department of Electrical and Computer Engineering\\
%Isfahan University of Technology\\
%84156-83111, Isfahan, Iran}%\fnref{label3}}
%\fntext[label3]{label3}

%{\small
%\author{
%{ Ali Ghiasian}$^a$, { Hossein Saidi}$^a$ and { Behnaz Omoomi}$^b$\\
%[1mm]}
%$^a${\small \it  Department of Electrical and Computer Engineering}\\
%{\small \it  Isfahan University of Technology} \\
 %{\small \it  Isfahan, Iran} \\
%$^b${\small \it  Department of Mathematical Sciences}\\
%{\small \it  Isfahan University of Technology} \\
%{\small \it 84156-83111, Isfahan, Iran}}

%@@@@@@@@@@@@@@@@@@@@@@@@@@@@@@@@@@@@@@@@@@@@@@@@@@@@@@

 %\maketitle \baselineskip15truept

\begin{abstract}
In single channel wireless networks, concurrent transmission at different links may interfere with each other. To improve system throughput, a scheduling algorithm is necessary to choose a subset of links at each time slot for data transmission. Throughput optimal link scheduling discipline in such a wireless network is generally an NP-hard problem. In this paper, we develop a polynomial time algorithm for link scheduling problem provided that network conflict graph is line multigraph. (i.e., line graph for which its root graph is multigraph). This result can be a guideline for network designers to plan the topology of a stationary wireless network such that the required conditions hold and then the throughout optimal algorithm can be run in a much less time. \\

\end{abstract}

\footnotetext[1]{This work was supported by Iran Telecom Research Center (ITRC)}

%%%%%%%%%%%%%%%%%%%%%%%%%%%%%%%%%%%%%%%%%%%%%%%%%%%%%%%%%%%%%%%%%%%%%%%%%%%%%%%%%%%%
%%%%%%%%%%%%%%%%%%%%%%%%%%%%%%%%%%%%%%%%%%%%%%%%%%%%%%%%%%%%%%%%%%%%%%%%%%%%%%%%%%%%%%%
\section{Introduction}

The underlying wireless network is shown by an undirected and connected graph $G(V,E)$ in which $V$ is the set of vertices and
$E$ is the set of edges. Every node of the network is represented by a vertex in graph $G$.
Two vertices are adjacent if they are within  communication range of each other. We assume that time is slotted. In single channel wireless networks, concurrent transmission at the same time slot and different links (edges) may interfere with each other. Therefore, a scheduling discipline is necessary to choose a subset of links at each time slot such that
packets do not corrupt due to the interference. 

Depending on the method used to deal with interference in such a radio network,
different models have been introduced in the literature. A general approach to deal with interference is to consider a {\it conflict graph}. The conflict graph of a given graph $G(V,E)$ is graph $G^c(E,L)$. Each vertex in $G^c$ is corresponding
to an edge in $G$, and two vertices in $G^c$ are adjacent whenever their corresponding edges in $G$ are interfering edges.
In this approach, when a link is ready for transmission, only a subset of links which are called the {\it interference set} needs to be considered as interfering links.
In other words, each link is associated with an interference set such that the
link can be scheduled only if no other link in its interference set is scheduled. Note that if link $l_1$ interferes with link $l_2$ then $l_2$ interferes with $l_1$ as well.
  Finding a set of non interfering links in $G$ is the same as finding an {\it independent set} in $G^c$. An Independent set in a graph is a collection of vertices such that there are no edges between them. 
  
  We shortly describe how the so called conflict graph can be constructed based on general {\it M-hop interference model}.
First, we refer to some more terminologies of graph theory which we use throughout the paper. 
The {\it distance}  between two vertices $u$ and $v$
in a graph $G$, denoted by $d_G(u,v)$, is the length of a
shortest path between $u$ and $v$ in $G$. The distance between
two edges is defined as a function $d: (E,E) \longrightarrow
\mathbb{N}$,   such that for every two edges $u_1u_2$ and $v_1v_2$,
$d(u_1u_2, v_1v_2)=\min_{\{i,j\}\in\{1,2\}} d_G(u_i,v_j)$. The
{\it power} of a graph $G$, denoted by $G^t$, $t \in \mathbb{N}$, is a graph with the same
set of vertices as $G$ in which two vertices $u$ and $v$ are
adjacent in $G^t$ if and only if $d_G(u,v)\leq t$. A {\it loop}  in graph is an edge that connects a vertex to itself. {\it Multiple edges} are two or more edges that are incident to the same two vertices. A {\it multigraph}  is a graph with multiple edges. A {\it simple}  graph is a graph without
loops and/or multiple edges. An {\it edge contraction}  is an operation which removes an edge from a graph while simultaneously merging its   end  vertices. 
We refer to~\cite{west2000} for other graphical notations and terminologies not
described in this paper.

The {\it line graph } of a graph $G=(V,E)$, denoted by $L(G)$, is a
graph with vertex set $E$, where two vertices of $L(G)$ are
adjacent if their corresponding  edges in $G$ are adjacent, i.e.
they have a common end vertex. In this case, we call graph $G$ the
{\it root graph} of $L(G)$. A graph $G$ is called a line graph if there is a root graph $G'$ such that $G=L(G')$.

Following the definition of line graph, we now introduce {\it M-hop interference model} ~\cite{sharma2006} which is mostly used to construct the conflict graph. Under this general interference model, two edges $l_1$ and $l_2$ are interfering edges if $d(l_1,l_2) \leq M$.  Therefore, the conflict graph can be defined as follows,

\begin{align} \label{eq2}
G^c (E,L)=[L(G)]^M, \quad M\geq 1.
\end{align}

This general interference model is applicable for extensive number of practical applications such as Bluetooth, FH-CDMA systems, Wireless LAN (IEEE 802.11 standard), etc. ~\cite{sharma2006,yi20082}. More details about different interference models can be found in \cite{santi2009}. For example, in Bluetooth and FH-CDMA systems,  two adjacent  edges  are interfering edges. Link scheduling in these networks results in finding a {\it matching} in graph $G$. 
A matching in a graph is a set of edges with no common end vertices. 

In IEEE 802.11 wireless LAN network under the RTS/CTS scheme, two edges that are either adjacent or  are both incident on a common edge are interfering edges. Link scheduling in this network results in finding a strong matching in graph $G$. A matching is called {\it strong matching} if no edge connects two edges of the matching~\cite{golumbic2000}.

Note that in IEEE 802.11 wireless LAN networks, the conflict graph can be constructed by using Eq.(\ref{eq2}) and setting $M=2$ while in Bluetooth networks the conflict graph is the same as $L(G)$ which is derived by setting $M=1$ in Eq.(\ref{eq2}).

Link scheduling algorithms are of interest due to their impact on the network throughput. Throughput optimal algorithms have been studied extensively in the literature \cite{brzezinski2008, gupta20102, modiano2006, ni2009, tassiulas1992, yi20081, zussman2008}. Let assume that associated to each link is a queue and packets are queued before they are transmitted over the link. A well known throughput optimal link scheduling algorithm is to find maximum weight independent set $(MWIS)$ at each time slot in the conflict graph, where the weight of each vertex is defined as the queue length of its corresponding link in the network graph. 

Finding the $MWIS$ is one of the known NP-Hard problems in graph theory~\cite{sharma2006}. However, if the conflict graph is line graph, then finding $MWIS$ in $G^c$ equals finding maximum weight matching ($MWM$) in its root graph. Since there are polynomial time complexity algorithms for $MWM$ problem \cite{lawler2001}, then the overall solution is much simpler under this assumption. The key point here that catches our attention is that the root graph does not required to be simple graph. If the root graph is multigraph, it is enough to keep the heaviest edge among multiple edges and remove the others before running $MWM$ algorithm. 

Following to this motivation for studying the line graphs, we explore these kind of graphs more precisely. Line graphs are well characterized class of graphs. In \cite{beineke1968} it is
proved that a graph $G$ is a line graph of a simple graph $G'$ if
and only if $G$ does not contain any of the forbidden nine graphs, depicted
in Figure~1, as an {\it induced subgraph}. An induced subgraph of a graph is a subset of vertices
of the graph  with edges whose endpoints are both in this subset. Whitney proved that with two
exceptional case (triangle and star with three branches, $G_1$ in Figure~1) the
structure of $G'$ can be recovered completely from its line graph~\cite{west2000}.

\begin{figure*}[t!]
\centering
\label {fig1}\includegraphics[trim= 25mm 210mm 5mm 10mm,clip,scale=.7]{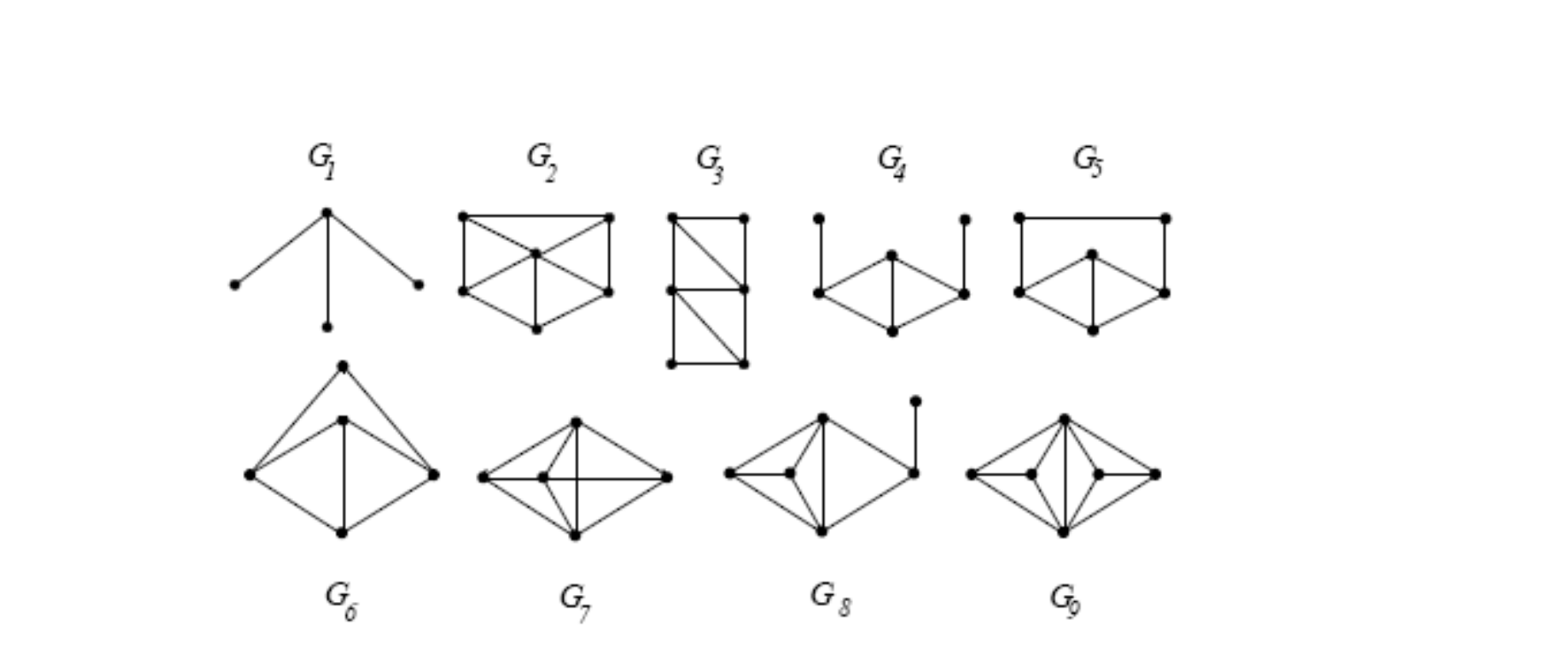}
\caption{Nine minimal forbidden graphs.}
\end{figure*}

It is worth mentioning that Lehot has developed an optimal
algorithm which can be run in linear time to detect whether a
graph is line graph and beget its root graph~\cite{lehot1974}.
Lehot algorithm considers only simple graphs as root graphs.

In this paper, following our motivation that allows the root graph to be multigraph, we introduce a generalization of line graph to {\it line multigraph}, i.e., line graph for which its root graph is
multigraph. Then, we extend Lehot algorithm to line multigraphs and
propose a low complexity algorithm, termed as extended Lehot ({\it eLehot}), for detecting whether a graph
is line multigraph and output its root graph. Accordingly, by allowing the root graph to be multigraph, we relax the constraint shown in Figure~1, to seven minimal forbidden graphs . Not only the number of forbidden graphs are reduced, prevention of them in the topology construction of the graph is much simpler because they are larger in the number of vertices and edges. 

The results of this paper introduce a new approach in topology control algorithms in wireless networks where the final target is complexity reduction. It complements the original motivation of topology control disciplines which tries to minimize energy consumption while the connectivity of network graph is guaranteed \cite{santi2005, wang2008}. Then, a new design dimension can be added to topology control algorithms by the results of this paper. As a result, based on available polynomial time complexity algorithms for $MWM$ problem \cite{lawler2001} and due to the linear time complexity of Lehot algorithm \cite{lehot1974}, we develop a polynomial time complexity approach for link scheduling algorithm under general M-hop interference model for the class of graphs that their conflict graphs are line multigraphs. 
In addition to topology control algorithms, the results of this paper can be used as a guideline for network designers when they want to design the topology of a stationary wireless network, e.g., positioning the routers/gateways of a wireless mesh network (WMN). If they prohibit the construction of derived forbidden graphs in the network's conflict graph, the throughput optimal link scheduling algorithm can be run in the network in much less time. Then the overall performance of the network is obviously promoted.

\begin{figure*}[t!]
\centering
\label {fig2}\includegraphics[trim= 25mm 250mm 5mm 10mm,clip,scale=.9]{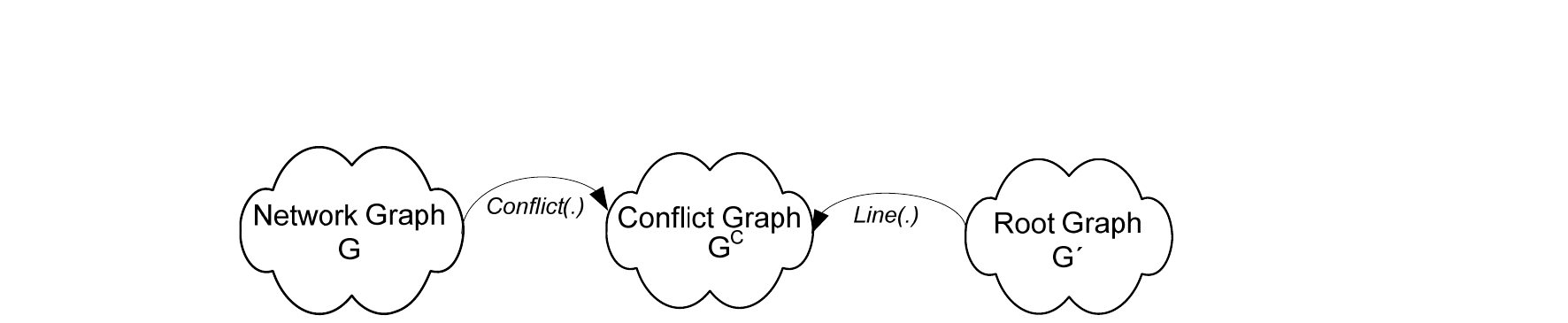}
\caption{Relation between network graph, conflict graph and root graph.}
\end{figure*}

Let us consider Figure~2 which depicts the idea. The number of
edges in $G'$ is equal to the number of edges in $G$, both equal
to the number of vertices in $G^c$. Note that, since there is a one to one mapping
between each vertex of $G^c$ and each edge in $G'$, then there is
also a one to one relation between edges in $G$ and edges in $G'$.
Also, note that finding a scheduling in $G$ is the same as finding
an independent set in $G^c$, while finding an
independent set in $G^c$ is equivalent to finding a 
matching in $G'$. Thus, we can deploy a policy of edge selection on
$G'$ to obtain an interfering free link selection in network graph $G$. 

The structure of this paper is as follows.  
In Section~2, we propose eLehot algorithm as an extension to Lehot algorithm.
We analyze eLehot's algorithm in Section~3. Finally, we conclude in Section~4.

%\begin{figure*}[t]
%\centering
%\label {fig3}\includegraphics[trim= 40mm 220mm 5mm 30mm,clip,scale=.9]{fig3}
%\caption{An example to clarify the main idea}
%\end{figure*}

\section{eLehot Algorithm}

Suppose that the graph $G^c$ is given and we want to find the root graph $G'$
such that $L(G') = G^c$ (if $G'$ exists), where $G'$ may be a
multigraph.
Two vertices $u$ and $v$ are called {\it true
twins} if they are adjacent and their neighborhoods are the same.
If two non-adjacent vertices have identical neighborhoods they are
called {\it false twins}. In the rest of the paper, where ever we use the term twin vertices, we mean true twin vertices.  
The following observation indicates that $k$ mutually
true twin vertices in $G^c$ are the vertices of a {\it clique}. A clique is a complete subgraph of a graph.

\begin{preobs} If vertices $u_1$ and $u_2$ are twins, and if $u_2$ and $u_3$ are twins as well, then $u_1$ and
$u_3$ are twins. If $u_1,u_2,\ldots,u_t$ are mutually twin vertices, then they are the vertices of a clique $K_t$,
$t>0$. 
\end{preobs}

We describe eLehot algorithm in Figure 3.

\begin{figure}[t!]
\begin{center}
\begin{boxit}
%\mbox{\hspace{120mm}}
\begin{tabbing} 
{\bf Algorithm} {\tt eLehot} \\
{\bf Input :} $G^c$\\
mm \=mmmm \=mmmm\= mmmm \=mmm \=mmm mmm \kill
{\bf Step 1.} Mark all edges in $G^c$ which their end vertices are twins. \\
\quad Then contract all marked edges.\\
\quad Label vertices with the number of contracted edges incident on it. \\
\quad Finally, consider the obtained simple vertex weighted graph as graph $H$.\\
\\
{\bf Step 2.} Run Lehot algorithm on the graph $H$.\\
\quad (refer to \cite{lehot1974} for the description of Lehot algorithm).\\
\\
{\bf Step 3.} If Lehot algorithm outputs the root graph, say $H'$, \\
\quad then equal to the weights of each vertex in $H$, \\
\quad add multiple edges to the corresponding edge in $H'$. \\
\quad The resulting graph is $G'$.\\
{\bf Output :} $G'$\\
\end{tabbing}
\end{boxit}
\caption{eLehot Algorithm.}
\end{center}
\label{fig3}
\end{figure}

Note that we do not care about the uniqueness of $G'$. In Figure 4, it is shown that using eLehot algorithm, the last six graphs of the nine forbidden subgraphs (Figure 1) are line multigraphs. The marked edges are denoted by the symbol $``//"$ in the figure.
In the first graph, we have plotted all the steps in the algorithm in  details, but for the others, the final result has been shown. 

\begin{figure*}[t!]
\centering
\label {fig4}\includegraphics[trim= 30mm 120mm 5mm 10mm,clip,scale=.75]{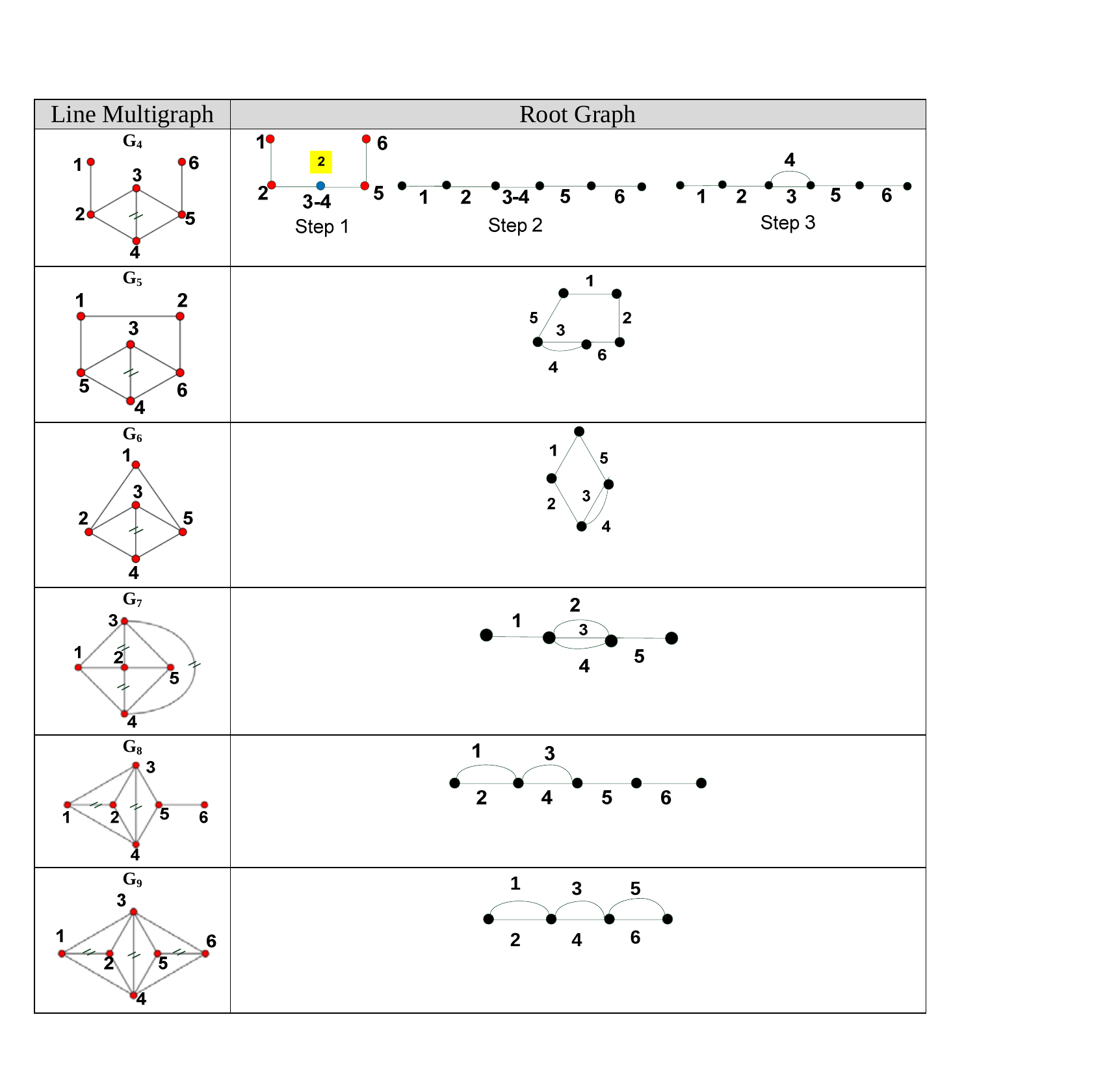}
\caption{Running eLehot algorithm on  the last six graphs in  Figure 1.}
\end{figure*}

\section{eLehot Algorithm Analysis}
In this section,  through three main theorems, we provide the necessary and sufficient condition for the  eLehot algorithm to have an output, prove the correctness of the algorithm and analyze it's complexity . First, we need to prove the following lemmas. 

% \subsection{Necessary and Sufficient Condition} 

\begin{lemma}After running Step~1 of eLehot algorithm, no true twin
vertices will remain or produce in the resulting graph. 
\end{lemma}

\begin{proof} {To see this fact, it is enough to show that for every two adjacent vertices $u$ and $v$ in the resulting graph $H$, there exists a vertex $u'$ adjacent to $u$ which is not adjacent to $v$. 

Since contraction operation does not create any new edges, edge $uv$ exists in $G^c$ and the vertices $u$ and $v$ are not twin in $G^c$.  Hence, there exists a vertex $u'$ in $G^c$ adjacent to $u$ and not adjacent to $v$. Thus, vertex $u'$ is the desired vertex in $H$.
}\end{proof}
Preposition~1 shows that running one round of contraction (Step 1) is sufficient. This property is required in the complexity analysis of eLehot algorithm.

\begin{lemma} For every  induced
subgraph $F$ in $H$,  there exists a twin less  induced subgraph in $G^c$ isomorphic to $F$ and vice versa. 
\end{lemma}

\begin{proof}{
 To see this, suppose that $F$ is an induced subgraph of $H$. First, we
show that there exists a subgraph $F$ in the main graph $G^c$.
According to Observation~1, each vertex of $F$ with multiplicity
$t$, is representative of a clique, $K_t$ in $G^c$. Now
to construct a subgraph $F$ in $G^c$, it is sufficient
to select one vertex from the cliques corresponding to the
vertices of $F$ and make the adjacency between these vertices the
same as the adjacency of the vertices in $F$ (Figure~5 clarifies
this approach). The obtained subgraph in $G^c$ is induced subgraph isomorphic to $F$, since the adjacency and non adjacency relation of the corresponding vertices are preserved in the contraction.

 Similarly, the vice versa of this process can be used to obtain the desired twin less induced subgraph of  $G^c$ in $H$. }
\end{proof}

\begin{figure*}[t]
\centering
\label{fig5}\includegraphics[trim= 40mm 230mm 5mm 20mm,clip,scale=1]{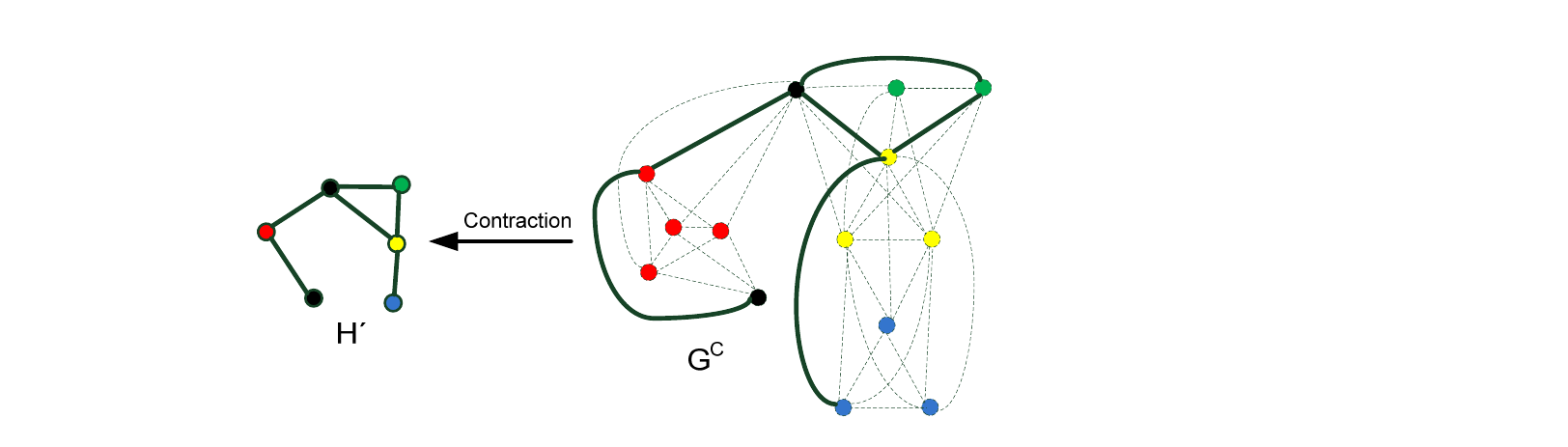}
\caption{Induced subgraphs of H can be found as induced subgraphs in $G^C$ (Colored vertices are contracted ones). }
\end{figure*}

%\begin{prepro} 
%The necessary and sufficient condition that graph $H$ does not contain any of nine forbidden graphs of Figure~\ref{fig1} as induced subgraph is that graph $G^c$ does not contain any of seven forbidden graphs of Figure~\ref{fig6}
%\end{prepro}

%\begin{proof}{
%The proof is given in Appendix A. }
%\end{proof}

\begin{prethm}
The eLehot algorithm has an output if and only if $G^c$ contains no induced subgraph $\{F_1, F_2, ..., F_7\}$ shown in Figure~6. 
\end{prethm}

\begin{figure*}[t]
\centering
\label{fig6}\includegraphics[trim= 30mm 210mm 5mm 10mm,clip,scale=.75]{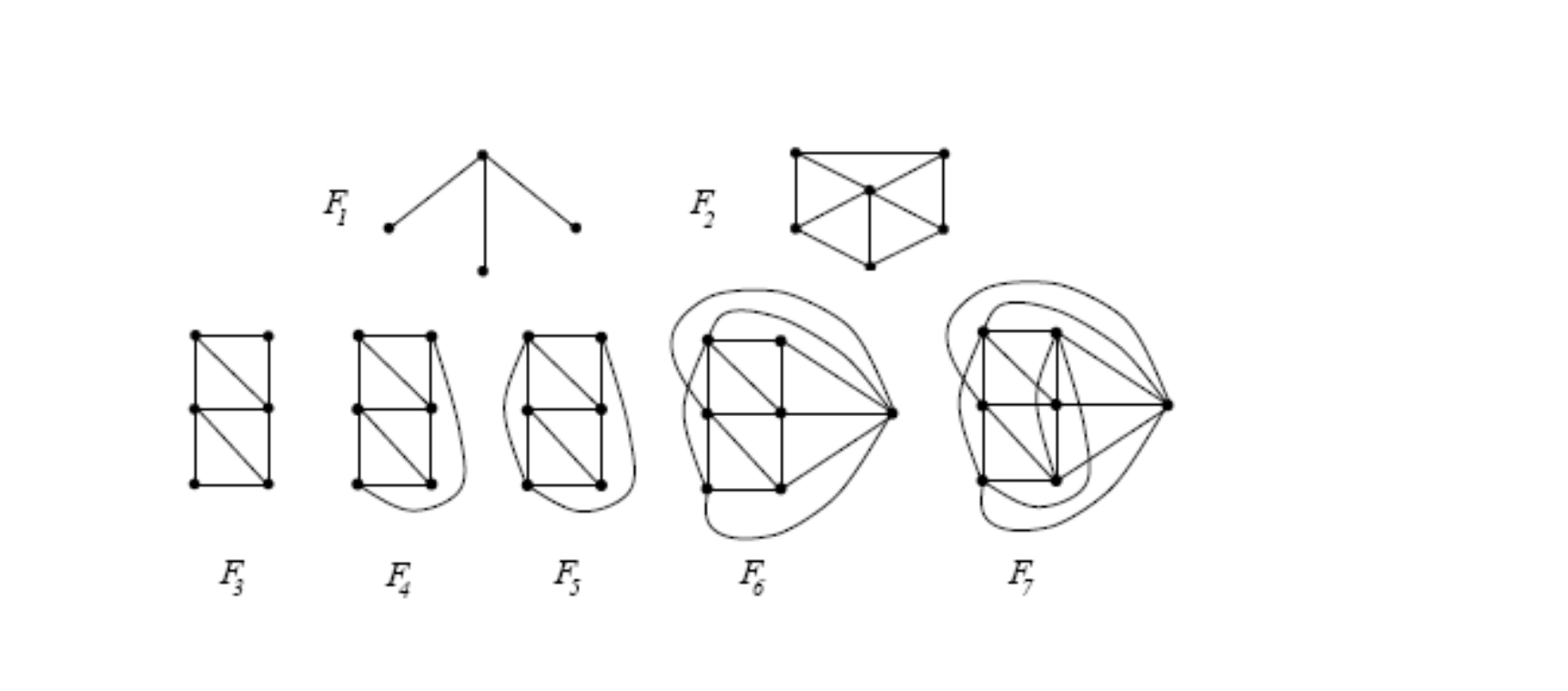}
\caption{Seven forbidden graphs of Theorem 1.}
\end{figure*}

\begin{proof}{
First, it is easy to see that the eLehot algorithm on $G^c$ has an output if and only if the Lehot algorithm on $H$ has an output. On the other hand, by Beineke's theorem \cite{beineke1968} it is known that Lehot algorithm has an output if and only if it's input graph contains no induced subgraph $\{G_1, G_2, ..., G_9\}$ shown in Figure~1.

Therefore, to prove the statement it is enough to see that graph $H$ contains an induced subgraph $\{G_1, G_2, ..., G_9\}$ if and only if the conflict graph $G^c$ contains an induced subgraph $\{F_1, F_2, ..., F_7\}$.

Note that, by Lemma~1, the resulting graph $H$ after running Step~1 of eLehot algorithm, removes all twin vertices in $G^c$ and don't produce new twin vertices. Hence, $H$ is a twin less graph. Also, by Lemma~2, if $F$ is an induced subgraph of $H$, then $G^c$ contains an induced subgraph isomorphic to $F$ and vice versa. 

We  divide the graphs of Figure 1 to two classes, $\mathcal{E}=\{G_1,G_2,G_3\}$, say twin less graphs, and $\mathcal{E'}=\{G_4,...,G_9\}$ that all have twin vertices as shown in Figure~4. 

First assume that $H$ contains one of the induced subgraphs $\{G_1, ..., G_9\}$. The key point that $H$ is a twin less graph leads us to examine graphs in $\mathcal{E'}$ one by one and for each of them show that how insertion of new neighbor vertex (or vertices) for one of the twin vertices can make the graph without twin vertices. The extracted minimal twin less subgraphs are the new forbidden graphs. This process shows that these new forbidden graphs are four graphs  $F_4, F_5, F_6, F_7$ in Figure~6.
 
Note that the above argument does not hold for graphs in  class $\mathcal{E}$. Since these graphs do not have any twin vertices, graph $H$ can be any of them and then they are still minimal forbidden graphs. Therefor the  minimal forbidden graphs for $H$ are three graphs of twin less class, $\mathcal{E}$, which are shown in Figure~6 by $F_1,F_2$ and $F_3$ in addition to four graphs that are derived from graphs in class $\mathcal{E'}$ and is depicted in Figure~6 by $F_4, F_5, F_6$, $F_7$. 

In what follows, we consider each of six graphs of class $\mathcal{E'}$ separately to see how we can make them twin less by adding the minimum number of vertices  to twin vertices. Whenever we encounter  to one of the  known forbidden (induced) subgraph, we terminate and go to the next case. We refer to Figure~7 to see the process of constructing the forbidden subgraphs. 

\noindent {\bf { i)} Consider the graph $G_4$}

Look at Figure 7-A1 The vertices 3 and 4 are true twins. To remove twin property, there should be another vertex $x$ that is adjacent to either 3 or 4. Due to the symmetry of $G_4$, we suppose that $x$ is adjacent to 4. The following options for adjacency of $x$ to other nodes are possible. Note that the symmetry of the graph helps us to eliminate similar cases and keep only one of them for investigation. 

If $x$ is only adjacent to the vertex 4, then the four vertices $\{4,2,5,x\}$ make graph $F_1$ (claw) regardless of adjacency of $x$ to 1 and 6. If $x$ is adjacent to the vertices 4 and 5, then the four vertices $\{5,3,x,6\}$ make a claw. If $x$ is adjacent to the vertices 4, 5 and 6 as shown in Figure 7-A1, then the resulting graph includes graph $F_3$ as induced subgraph (by deleting vertex 1). If $x$ is adjacent to the vertices 4,5,6 and 2, then $\{2,1,3,x\}$ is a claw.
If $x$ is adjacent to the vertices 4, 5, 6 and 1, then $\{x,1,4,6\}$ is a claw regardless of adjacency of $x$ to vertex 2. 
If $x$ is adjacent to the vertices 4, 5, 6, 2 and 1, then $\{x,1,4,6\}$ is a claw. (Figure 7-A2)

The above investigations show that $G_4$ can be removed from the list of forbidden graphs since prevention of $F_1$ and $F_3$ provides the same result.

\noindent {\bf { ii)} Consider the graph $G_5$}

Look at Figure 7-B1. The vertices 3 and 4 are  twin. To remove twin property, there should be another vertex $x$ that is adjacent to either 3 or 4. Due to the symmetry of $G_5$, we suppose that $x$ is adjacent to 4. The following options for adjacency of $x$ to other vertices are possible. Note that the symmetry of the graph helps us to eliminate similar cases and keep only one of them for investigation. 

If $x$ is only adjacent to vertex 4, then $\{4,2,5,x\}$ is a claw regardless of adjacency of $x$ to 1 and 6. 
If $x$ is adjacent to the vertices 4 and 5, then the four vertices $\{5,3,x,6\}$ make a claw. 
If $x$ is adjacent to the vertices 4, 5 and 6, then the resulting graph includes graph $F_3$ as induced subgraph (by deleting vertex $1$, in Figure 7-B1)
If $x$ is adjacent to the verices 4, 5, 6 and 2, then $\{2,1,3,x\}$ is a claw.
If $x$ is adjacent to the vertices 4, 5, 6 and 1, the resulting graph includes graph $F_3$ as induced subgraph (by deleting vertex $1$).
If $x$ is adjacent to the vertices 4,5,6,2 and 1, the resulting graph includes graph $F_2$ as induced subgraph (by deleting vertex $3$, in Figure 7-B2). 
The above investigations show that $G_5$ can be removed from the list of forbidden graphs since prevention of $F_1, F_2$ and $F_3$ provides the same result.

\noindent {\bf { iii)} Consider the graph $G_6$}

Look at Figure 7-C1. The vertices 3 and 4 are  twin. To remove twin property, there should be another vertex $x$ that is adjacent to either 3 or 4. Due to the symmetry of $G_4$, we suppose that $x$ is adjacent to 4. The following options for adjacency of $x$ to other vertices are possible. 

If $x$ is only adjacent to the vertex 4, then $\{4,1,2,x\}$ is a claw.
If $x$ is adjacent to the vertices 4 and 2, then $\{2,5,3,x\}$ is a claw.
If $x$ is adjacent to the vertices 4, 2 and 5, the achieved subgraph is shown in Figure 7-C1 and should be added to the list of forbidden graphs for line multigraphs since it is a new minimal graph that contains one of Beineke's forbidden graphs ($G_6$) as an induced subgraph and does not have any twin vertices. We call this graph as $F_4$ in Figure~6.

If vertex $x$ is adjacent to the vertices 4, 2, 5 and 1, the achieved subgraph is shown in Figure 7-C2 and should be added to the list of forbidden graphs for line multigraphs since it is a new minimal graph that contains one of Beineke's forbidden graphs ($G_6$) as induced subgraph and does not have any twin vertices. We call this graph as $F_5$ in Figure~6.

\noindent{\bf { iv)} Consider the graph $G_7$}

This graph contains three mutual twin vertices. Therefore, we need two extra vertices say $x$ and $y$ to remove twin property of the graph. All the adjacency possibilities that make the graph twin less are discussed in the following. 
Note that the symmetry of the twin vertices 3, 4 and 5 and the symmetry of vertices 1 and 2 in Figure 7-D1 helps us to abstract the possible options as follows. 

We should consider three cases. 

\noindent { Case 1.} Vertex $x$ is adjacent to vertex $5$ and vertex $y$ is adjacent to vertex $3$.

If  vertex $x$ is only adjacent to vertex $5$, then $\{5,x,1,2\}$ is a claw. The same occurs if vertex $y$ is only adjacent to one of twin vertices $3$ or $4$. Hence, vertices $x$ and $y$ should be adjacent to more than one vertex of $G_7$. 

If $x$ is adjacent to the vertices 5 and 1; and $y$ is adjacent to the vertices 3 and 1, then $\{1,x,y,4\}$ is a claw (Figure 7-D1).
If $x$ is adjacent to the vertices 5 and 1; $y$ is adjacent to the vertices 3 and 1; and $x$ and $y$ are adjacent, then graph $F_3$ is an induced subgraph of the achieved graph which is shown in Figure 7-D2 (remove vertex 3).

If $x$ is adjacent to the vertices 5 and 1; $y$ is adjacent to the vertices 3 and 2; then graph $F_3$ is an induced subgraph of the achieved graph which is shown in Figure 7-E1 (remove vertex 4).
If $x$ is adjacent to the vertices 5, 1 and 2; $y$ is adjacent to the vertices 3 and 1; $x$ and $y$ are adjacent, then the constructed graph contains graph $F_5$ as induced subgraph as shown in Figure 7-E2. The induced graph $F_5$ is achieved by removing vertex 4 and is redrawn in Figure 7-E3 for clarity. 

If $x$ is adjacent to the vertices 5, 1 and 2; $y$ is adjacent to the vertices 3 and 2; $x$ and $y$ are adjacent (Figure 7-F1), then the constructed graph contains graph $F_5$ as induced subgraph as shown in Figure 7-F2. The induced graph $F_5$ is achieved by removing vertex 4 and is redrawn in Figure 7-F3 for clarity. 

If $x$ is adjacent to the vertices 5, 1 and 2; $y$ is adjacent to the vertices 3, 1 and 2; $x$ and $y$ are adjacent (Figure 7-G1), then the constructed graph contains graph $F_5$ as induced graph as shown in Figure 7-G2. The induced graph $F_5$ is achieved by removing vertex 5 and is redrawn in Figure 7-G3 for clarity. 

\noindent{ Case 2.} Vertex $x$ is adjacent to the vertices 5 and 4; vertex  $y$ is adjacent to the vertex 4.

If $x$ is adjacent to the vertices 5 and 4; $y$ is adjacent to the vertex 4 (Figure 7-H1), then $\{4,x,y,2\}$ , $\{4,x,y,3\}$ , $\{4,x,y,1\}$ , $\{4,2,1,x\}$, $\{4,2,1,y\}$ and $\{5,1,2,x\}$ are different induced claws. Since none of $x$ and $y$ could be adjacent to the vertex 3 in this case, adjacency of $x$ to $y$ is mandatory to prohibit claw $\{4,x,y,3\}$, otherwise this claw exists in all the scenarios. Thus, in other situations under case 2, we suppose $x$ and $y$ are adjacent. Also, $x$ and $y$ should be adjacent to the vertices 1 and/or 2 to prohibit claws $\{4,2,1,x\}$ and $\{4,2,1,y\}$. These observations leads us to the following scenarios. 

If vertex $x$ is adjacent to the vertices 5, 4 and 1; $y$ is adjacent to the vertices 4, 1; $x$ and $y$ are adjacent (Figure 7-H2), then the constructed graph contains $F_3$ as induced subgraph by removing vertex 4. 
If $x$ is adjacent to the vertices 5, 4 and 2; $y$ is adjacent to the vertices 4 and 2; $x$ and $y$ are adjacent (Figure 7-H3), then the constructed graph contains $F_3$ as induced subgraph by removing vertex 4 which is shown in the same figure. 

If $x$ is adjacent to the vertices 5, 4 and 2; $y$ is adjacent to the vertices 4 and 1; $x$ and $y$ are adjacent (Figure 7-I1), then the constructed graph contains $F_2$ as induced subgraph by removing vertex 5 which is shown in the same figure. 
If $x$ is adjacent to the vertices 5, 4 and 1; $y$ is adjacent to the vertices 4 and 2; $x$ and $y$ are adjacent (Figure 7-I2), then the constructed graph contains $F_2$ as induced subgraph by removing vertex 5 which is shown in the same figure. 
If $x$ is adjacent to the vertices 5, 4 and 1; $y$ is adjacent to the vertices 4, 2 and 1; $x$ and $y$ are adjacent (Figure 7-I3), then the constructed graph contains $F_4$ as induced subgraph by removing vertex 4. 

If $x$ is adjacent to the vertices 5, 4, 1 and 2; $y$ is adjacent to the vertices 4 and 2; $x$ and $y$ are adjacent (Figure 7-J1), then this is a new graph that does not contain any of previously found forbidden graphs and then should be added to the list of forbidden graphs. We rearrange its illustration as shown in Figure 7-J2 and call it as graph $F_6$. 

If $x$ is adjacent to the vertices 5, 4, 1 and 2; $y$ is adjacent to the vertices 4 and 1; $x$ and $y$ are adjacent (Figure 7-K1), then the constructed graph is the same as graph $F_6$.
If $x$ is adjacent to the vertices 5, 4 and 2; $y$ is adjacent to the vertices 4, 1 and 2; $x$ and $y$ are adjacent (Figure 7-K2), then the constructed graph contains $F_4$ as induced subgraph by removing vertex 4 (Figure 7-K3). 

If $x$ is adjacent to the vertices 5, 4, 1 and 2; $y$ is adjacent to the vertices 4, 1 and 2; $x$ and $y$ are adjacent (Figure 7-L1), then the constructed graph contains $F_5$ as induced subgraph by removing vertex 4 which is shown in Figure 7-L2. 

\noindent { Case 3.} Vertex $x$ is adjacent to the vertices 5 and 4; vertex $y$ is adjacent to the vertices 3 and 4.

If vertex  $x$ is adjacent to the vertices 5 and 4; $y$ is adjacent to the vertices 3 and 4 (Figure 7-M1), then $\{4,x,y,2\}$, $\{4,x,y,1\}$, $\{4,2,1,x\}$, $\{4,2,1,y\}$, $\{5,1,2,x\}$ and $\{3,y,1,2\}$ are different induced claws. We investigate the scenarios in which the mentioned claws does not exists. We first study the options that $x$ and $y$ are not adjacent and then consider the cases that $x$ and $y$ are adjacent.

If vertex $x$ is adjacent to the vertices 5, 4, 1; $y$ is adjacent to the vertices 3, 4, 2 (Figure 7-M2), then the constructed graph contains $F_3$ as induced subgraph by removing vertex 4 which is shown in Figure 7-M3. 

If vertex $x$ is adjacent to the vertices 5, 4, 1; $y$ is adjacent to the vertices 3, 4, 2,1 (Figure 7-N1), then the constructed graph is $F_6$ as shown in Figure 7-N2 by rearranging the position of vertices. 

If vertex $x$ is adjacent to the vertices 5, 4, 2; $y$ is adjacent to the vertices 3, 4, 1 (Figure 7-O1), then the constructed graph contains $F_3$ as induced subgraph by removing vertex 4 which is shown in Figure 7-O2. 

If vertex $x$ is adjacent to the vertices 5, 4, 2; $y$ is adjacent to the vertices 3, 4, 1,2 (Figure 7-P1), then the constructed graph is isomorphic to graph $F_6$ as shown in Figure 7-P2.

If vertex $x$ is adjacent to the vertices 5, 4, 1, 2; $y$ is adjacent to the vertices 3, 4, 1 (Figure 7-Q1), then the constructed graph is isomorphic to graph $F_6$ as shown in Figure 7-Q2.

If vertex $x$ is adjacent to the vertices 5, 4, 1, 2; $y$ is adjacent to the vertices 3, 4, 2 (Figure 7-R1), then the constructed graph is isomorphic to graph $F_6$ as shown in Figure 7-R2.

If vertex $x$ is adjacent to the vertices 5, 4, 1, 2; $y$ is adjacent to the vertices 3, 4, 1, 2 (Figure 7-S1), then the constructed graph contains $F_5$ as induced subgraph as shown in Figure 7-S2.

If vertex $x$ is adjacent to the vertices 5, 4, 1; $y$ is adjacent to the vertices 3, 4, 2 and $x$ is adjacent to $y$ (Figure 7-T1), then the constructed graph contains $F_4$ as induced subgraph by removing vertex 4 which is shown in Figure 7-T2. 

If vertex $x$ is adjacent to the vertices 5, 4, 1; $y$ is adjacent to the vertices 3, 4, 2, 1 and $x$ is adjacent to $y$ (Figure 7-U1), then the constructed graph contains $F_5$ as induced subgraph which is shown in Figure 7-U2 by removing vertex 4. 

If vertex $x$ is adjacent to the vertices 5, 4, 2; $y$ is adjacent to the vertices 3, 4, 1 and $x$ is adjacent to $y$ (Figure 7-V1), then the constructed graph contains $F_4$ as induced subgraph by removing vertex 4 which is shown in Figure 7-V2. 

If vertex $x$ is adjacent to the vertices 5, 4, 2; $y$ is adjacent to the vertices 3, 4, 1,2 and $x$ is adjacent to $y$ (Figure 7-W1), then the constructed graph contains $F_5$ as induced subgraph which is shown in Figure 7-W2.

If vertex $x$ is adjacent to the vertices 5, 4, 1, 2; $y$ is adjacent to the vertices 3, 4, 1 and $x$ is adjacent to $y$(Figure 7-X1), then the constructed graph contains $F_5$ as induced subgraph which is shown in Figure 7-X2.

If vertex $x$ is adjacent to the vertices 5, 4, 1, 2; $y$ is adjacent to the vertices 3, 4, 2 and $x$ is adjacent to $y$ (Figure 7-Y1), then the constructed graph contains $F_5$ as induced subgraph which is shown in Figure 7-Y2.

If vertex $x$ is adjacent to the vertices 5, 4, 1, 2; $y$ is adjacent to the vertices 3, 4, 1, 2 and $x$ is adjacent to $y$ (Figure 7-Z1), then this is a new graph that does not contain any of previously found forbidden graphs and then should be added to the list of forbidden graphs. We call it as graph $F_7$.

\noindent{\bf { v)} Consider the graph $G_8$}

This graph contains two couple of twin vertices, nevertheless one extra vertex say $x$ suffices to remove twin property of the graph. It is because the twin vertices, unlike the graph $G_7$, do not share any common vertex and are completely separated couples. Meanwhile, due to the symmetry of the graph, only one possible solution for removing twin property of the graph exists which is shown in Figure 7-$\Theta 1$. The obtained graph is not a new graph since it contains claw $\{4,5,x,2\}$. 
Indeed any of the vertices 3 or 4 which makes adjacency with $x$  (vertex 4 in this figure), in addition to one vertex out of the set of vertices $\{1,2\}$ which is not adjacent to vertex $x$ (vertex 2 in the figure) in addition to vertex 5 and $x$ always make a claw. 
To prohibit the resulted claw, we consider the case that vertex $x$ is adjacent to vertex 5 too. Then a new claw, $\{5,6,3,x\}$ is made. Then the only possibility to prohibit this claw is making vertex $x$ adjacent to vertex 6. The derived graph is shown in Figure 7-$\Theta 2$. The result is graph $F_3$ which is depicted in Figure 7-$\Theta 3$ by rearranging the illustration of Figure 7-$\Theta 2$. 

\noindent {\bf { vi)} Consider the graph $G_9$}

This graph contains three couple of twin vertices, nevertheless one extra vertex say $x$ suffices to remove twin property of the graph. It is because the twin vertices, do not share any common vertex and are completely separated couples. Meanwhile, due to the symmetry of the graph, only one possible solution for removing twin property of the graph exists which is shown in Figure 7-$\Pi 1$. The obtained graph is not a new graph since it contains the claw $\{3,6,x,2\}$. Note that this claw could not be prohibited by connecting vertex $x$ to neither vertex 6, nor vertex 2, otherwise a twin couple is constructed again. Indeed any of the vertices 3 or 4 (vertex 3 in this figure) which makes adjacency with vertex $x$ in addition to one vertex out of each other twin vertices which is not adjacent to $x$ (vertices 6 and 2 in the figure) always make a claw. Therefore, graph $G_9$ does not result in a new forbidden graph. 

To complete the proof of Theorem~1, note that regarding to construction of the graphs $F_1,F_2,...,F_7$, it can be seen that every graph $F_i$, $1\leq i\leq 7$, is a twin less graph contains one of the induced subgraphs $G_1,G_2,...,G_9$. Thus, if $G^c$ contains one of the induced subgraph $F_i$, $1\leq i\leq 7$, then by Lemma~2 its induced subgraph $G_1,G_2,...,G_9$ preserves in graph $H$. 
}\end{proof} %theorem 1

%\subsection{Correctness}
In the following theorem, we prove that eLehot algorithm begets the root graph of  the conflict graph $G^c$.

\begin{prethm}
If graph $G'$ is the output of eLehot algorithm on conflic graph $G^c$, then $G'$ is the root graph of conflict graph $G^c$, i.e. $L(G')=G^c$.
\end{prethm}

\begin{proof}{
Asuume that $G'$ is the output of the eLehot algorithm and $H'$ be a simple graph obtained by $G'$ after removing the multiple edges of $G'$ but keeping one edge. Note that in Step 3 of eLehot algorithm, we make the multiple edges according to the label of vertices in $H=L(H')$. We keep the multiplicity of each edge as a label of its corresponding vertex in $H$. 
Therefore, the line graph of $G'$ is a graph obtained from $H$ by replacing every vertex by a clique of the size of its label. According to Step~1, this graph is the original graph $G^c$ as desired. 
}\end{proof}

By Theorems 1 and 2, we have the following corollaries. 

\begin{precor}
If the conflict graph $G^c$ contains no induced subgraph $F_1,F_2,...,\linebreak F_7$,  then $G^c$ is the line multigrah of graph $G'$, where $G'$ is an output of eLehot algorithm.
\end{precor}

\begin{precor}
A given graph G is a line multigraph if and only if eLehot algorithm has an output for it.
\end{precor}

\begin{proof}{
The necessity is the result of Theorem~2.
To see the sufficiency,  assume that  G is a line multigraph. In~\cite{bermond1973, hemminger1972}  it has been shown that a line multigraph does not contain any subgraphs of $\{F_1, F_2, ..., F_7\}$ as induced subgraph.  Moreover, by Theorem~1 we know that if G contains no graphs of $\{F_1, F_2, ..., F_7\}$ as induced subgraph, then eLehot algorithm has an output.
}\end{proof}
%\subsection{Complexity Analysis}
 
We analyze the algorithm's complexity in the following theorem. 
 
\begin{prethm}
The time complexity of eLehot algorithm is $O(|E|^3)$.
\end{prethm}

\begin{proof}{
According to Eq. (\ref{eq2}), the time complexity of constructing
$G^c (E,L)$ from $G(V,E)$ is $O(|E|^2)$. Running Step~1 of the
algorithm has complexity of $O(|L|.|E|)$ or $O(|E|^3)$,
considering that in the worse case maximum of $|L|$ could be up to
${{|E|(|E|-1)}\over 2}$. According to the time complexity
of Lehot algorithm, Step~2 has the complexity
$O(|E|^2)+O(|E|)$~\cite{lehot1974}. Step~3 of the
proposed algorithm has the complexity of $O(|E|)$.
Consequently, the overall process of
constructing $G'$ has polynomial time complexity of $O(|E|^3)$. 
%Since the complexity of $MWM$ algorithm is $O(|V|^3)$ \cite{lawler2001} the proposed approach of link scheduling in wireless networks incurs a time complexity of $O(|E|^3)$. 
}\end{proof}

\section{Conclusions and Future work}

In this paper, we have generalized the concept of line graphs to line multigraphs and applied it to the conflict graph of stationary wireless networks for the purpose of link scheduling. It is shown that applying $MWM$ algorithm on the root graph of the conflict graph is equivalent to the link scheduling under general M-hop interference model in the network graph. We have proposed an algorithm 
to detect whether the conflict graph is line multigraph and output its root graph. The proposed algorithm is an extension of the well known Lehot algorithm to the line multigraphs and is called eLehot.

While applying the throughput optimal link scheduling algorithm in general is an NP-Hard problem, our overall proposed method results in a low complexity polynomial time algorithm, provided that the conflict graph is line multigraph. It was shown that how the derived conditions can be satisfied by network designers through topology control of the network by prohibiting the construction of seven forbidden graphs in the conflict graph. We believe that the results of this paper can be used as a guideline for network designers to plan the topology of a stationary wireless network such that the required conditions hold and then the throughout optimal algorithm can be run in a much less time.  As a future plan, we aim to design a topology control algorithm based on the results of this paper. 

%\appendix
%\begin{landscape}
\begin{figure*}[t]
\centering
\label {fig7}\subfloat {\includegraphics[trim= 10mm 35mm 5mm 25mm,clip,scale=.42]{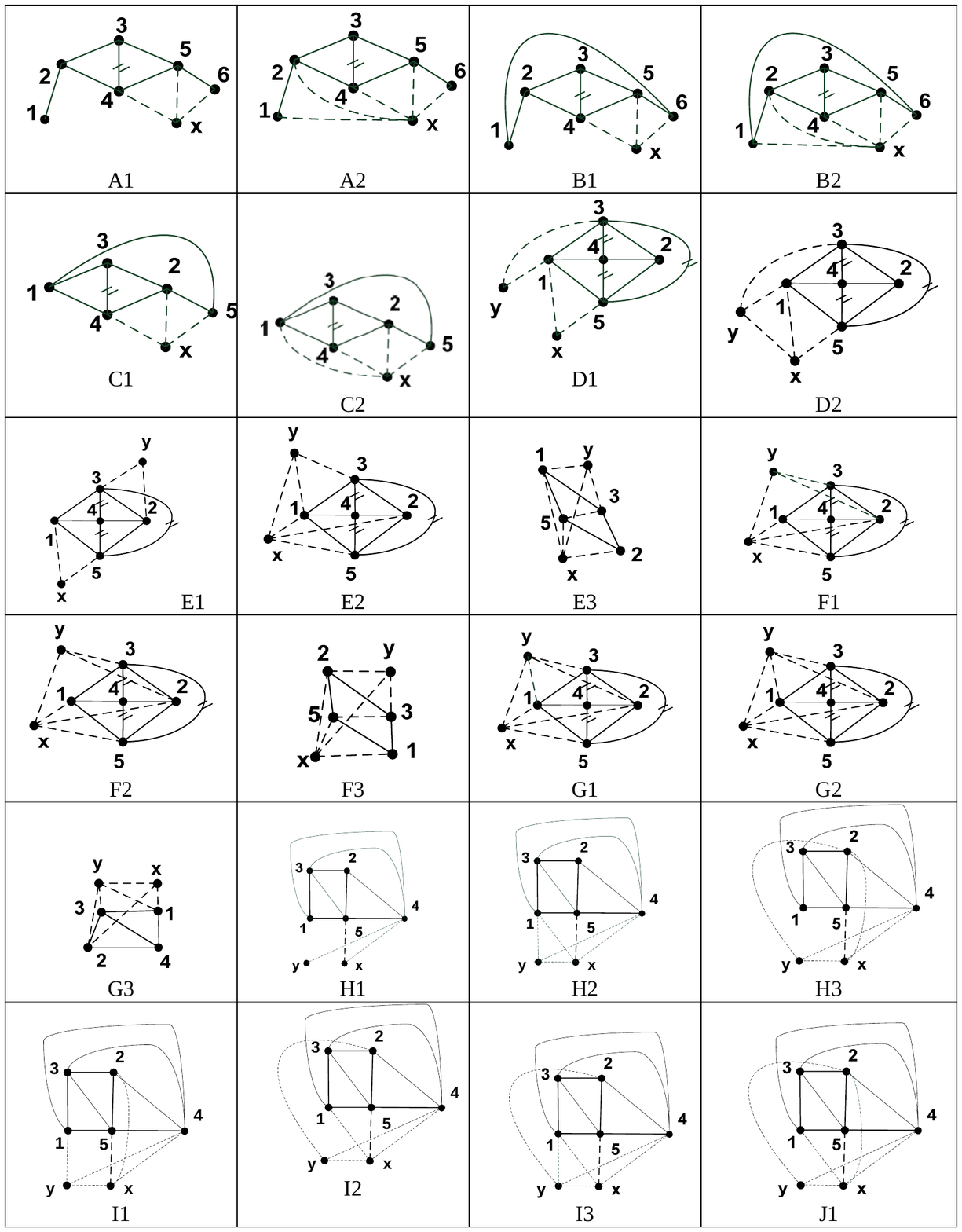}}
\subfloat {\includegraphics[trim= 10mm 28mm 5mm 25mm,clip,scale=.4]{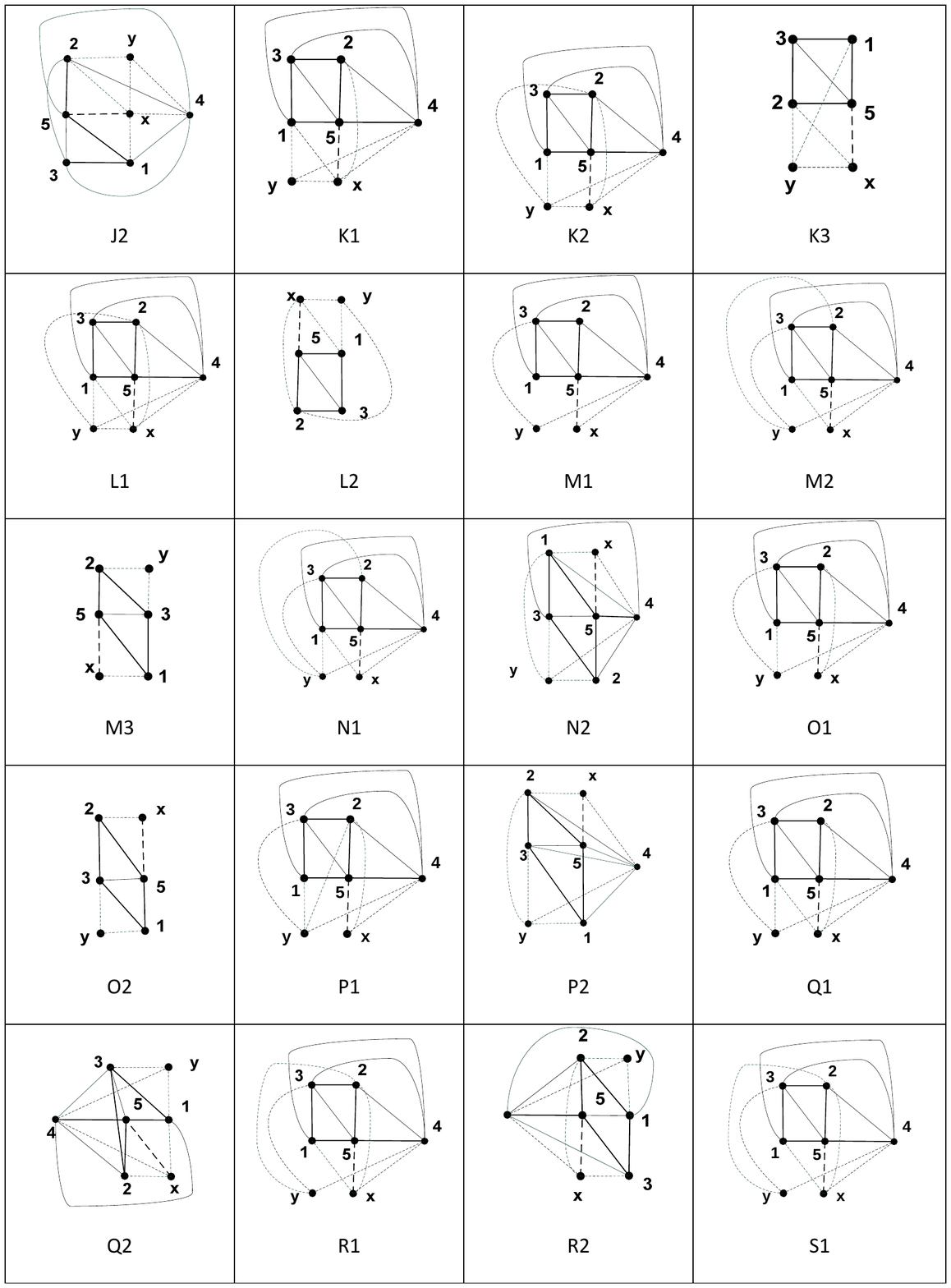}}\\
\subfloat {\includegraphics[trim= 10mm 28mm 5mm 20mm,clip,scale=.4]{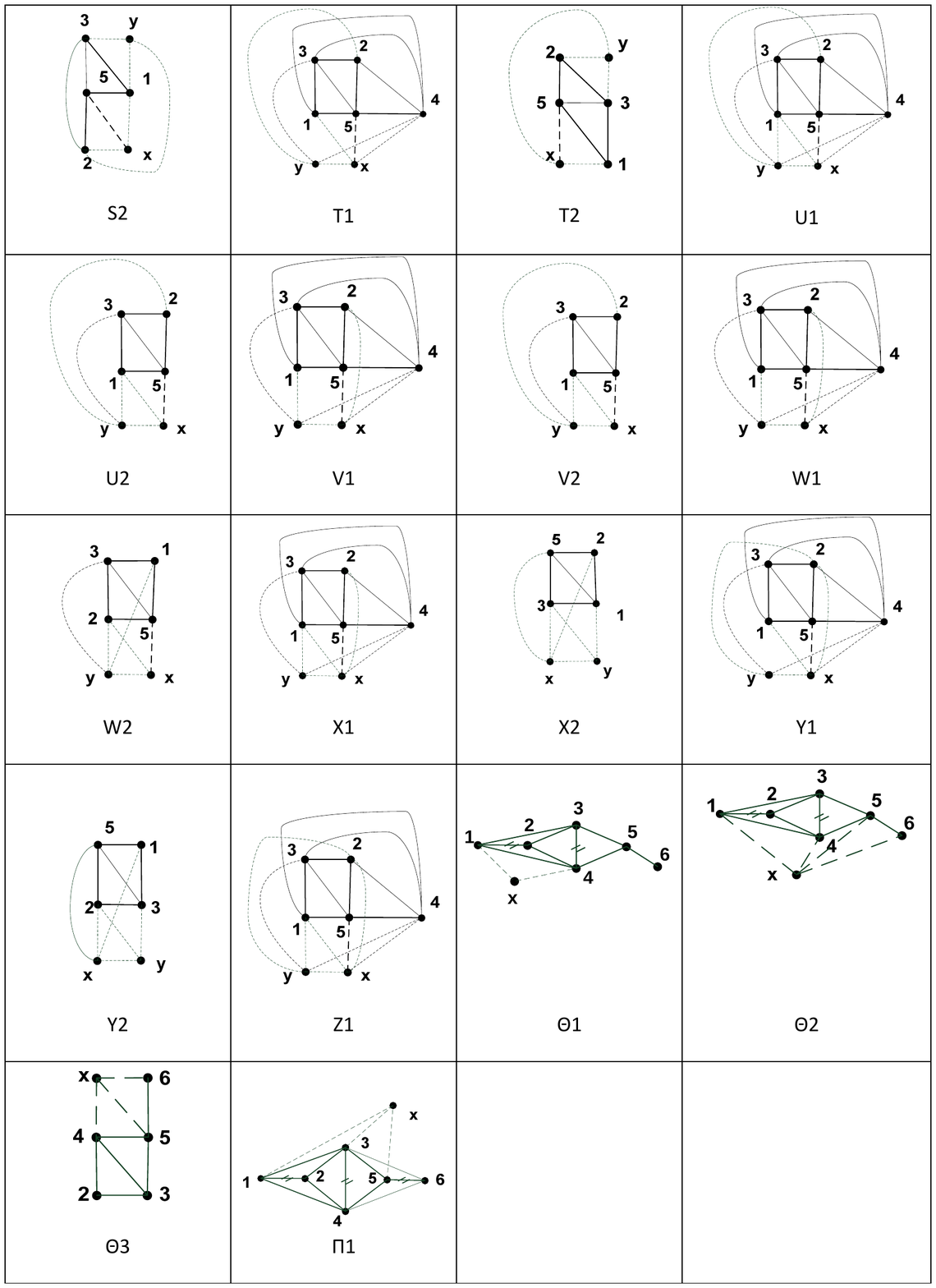}}
\caption{Construction of Forbidden Graphs.}
\end{figure*}
%\end{landscape}

%\biboptions{compress}
%\bibliographystyle{elsarticle-num-names}{}
%\bibliographystyle{model5-names}{}
%\bibliography{D:/Users/Ali/PhD/Wireless/JournalPapers/ref}   
  
%\biboptions{compress}
%\bibliographystyle{elsarticle-num-names}{}
%\bibliographystyle{plain}

%\bibliography{ref}
\end{document}